\documentclass[twoside, 12pt]{article}
\usepackage{graphicx}
\usepackage[russian, english]{babel}
\usepackage{epsfig}
\usepackage{amssymb,amsmath,amsfonts,soul,amsthm,enumerate}
\usepackage{amsfonts}
\usepackage{amsmath}
\usepackage{graphicx}
\usepackage{amsthm}
\usepackage[cp1251]{inputenc}
\usepackage[T2A]{fontenc}
\usepackage{mathrsfs}
\usepackage[english]{babel}
\newtheorem{theorem}{Theorem}[section]
\newtheorem{lemma}{Lemma}[section]

\newtheorem{remark}{Remark}

\numberwithin{equation}{section}
 \def\@evenhead{\vbox{\hbox to \textwidth{\thepage\hfil\sl\leftmark\strut}\hrule}}
 \def\@oddhead{\vbox{\hbox to \textwidth{\rightmark\hfill\thepage\strut}\hrule}}

\begin{document}
 \sloppy
 
 \centerline{\bf ON THE WEIGHTED HARDY TYPE INEQUALITY} \centerline{\bf FOR
FUNCTIONS FROM $W^1_p$} \centerline{\bf VANISHING ON SMALL PARTS OF THE BOUNDARY}    % Title of talk

\vskip 0.3cm

\centerline{\bf Yu.O. Koroleva}        % Authors from the same institution

\markboth{\hfill{\footnotesize\rm   Yu.O. Koroleva }\hfill}
{\hfill{\footnotesize\sl  On the Weighted Hardy Type Inequality  for
Functions from $W^1_p$ Vanishing on Small Parts of the Boundary}\hfill}
\vskip 0.3cm

\vskip 0.7 cm


\vskip 0.2cm


\vskip 0.2cm

\noindent {\bf Abstract.} A new  weighted
Hardy-type inequality for functions from
the Sobolev space $W_{p}^{1}$ is proved. It is assumed that functions vanish on small alternating pieces of the
boundary. The proved inequality generalizes the classical known weighted Hardy-type inequalities.

\section{\large Introduction}

Mathematical analysis needs to treat integral inequalities of different types. One of the useful estimate is the classical Hardy-type inequality:
\begin{equation}\label{hardy-1}
\int_{\Omega}\left(|u(x)|^\frac{np}{n-p}\right)^p\,dx\leq C(p,n)\left(\int_{\Omega}{|\nabla u(x)|}^p\,dx\right)^{\frac{1}{p}},
\end{equation}
where $u\in C_{0}^{\infty}(\Omega), \ p>1, n\in \mathbb{Z}_{+}, \ \Omega\in \mathbf{R}^n.$

This estimate could be generalized to the multidimensional
weighted form:
\begin{equation}\label{hardy-general}
\left(\int\limits_{\mathbb{R}^n}V(x)|u(x)|^q\,dx\right)^{\frac{1}{q}}\leq
{C}\left(\int\limits_{\mathbb{R}^n}W(x)|\nabla
u(x)|^p\,dx\right)^{\frac{1}{p}},
\end{equation}

 where $\ u(x)\in
C_{0}^{\infty}(\mathbb{R}^n),$ $V(x)\geq 0, W(x)\geq 0,\,p,q\geq 1,$
and the constant ${C}$ depends only on weight functions $V(x)\mbox{ and
}W(x).$ There are several results concerning weighted Hardy-type
inequality (see e.g. the books \cite{KMP}, \cite{KP} and \cite{OK}
and the references given there). The following result is by  V.\,Maz'ya (see
\cite[Corollary of Theorem 1.4.1.2, Theorem 1.4.2.2 ]{mazja}):

\begin{theorem}
Let $1<p<q<\infty,\ p<n $ or $1=p\leq q<\infty.$ Then the Hardy
inequality (\ref{hardy-general}) with $W(x)\equiv 1$ holds for every
$u\in C_{0}^{\infty}(\mathbb{R}^n)$ with a finite constant $C>0$ if
and only if
$$B:=\sup\limits_{x\in \mathbb{R}^n} \sup\limits_{R>0}R^{1-\frac{n}{p}}\left(\int\limits_{B_{R(x)}}V(y)\,dy\right)^\frac{1}{q}<\infty,$$
where $B_{R(x)}$ is a ball of the radius $R$ centered at the point
$x.$
\end{theorem}

J.\, Ne\v{c}as has generalized the weighted Hardy-type inequality to domains in
$\mathbb{R}^n$ in \cite{necas}.
He proved that if $\Omega$ is a bounded domain with Lipschitz
boundary, $1<p<\infty,\,$ $\alpha<p-1,$ then for $u\in
C^{\infty}_{0}(\Omega)$ the inequality
\begin{equation}\label{dist}
\int\limits_{\Omega}|u(x)|^{p}\varrho^{-p+\alpha}(x)\,dx \leq C
\int\limits_{\Omega}|\nabla u(x)|^{p}\varrho^{\alpha}(x)\,dx
\end{equation}
holds, where $\varrho(x)=dist(x, \partial \Omega).$ After that this
inequality was generalized by A.\,Kufner in \cite{Kufner} to domains
with the H\"older boundary and later by A.\,Wannebo (see
\cite{wannebo}) to domains with the generalized H\"older condition.
All results related to (\ref{dist}) in the case $\alpha=0$ was
described in \cite{KinnKorte}.

The aim of this paper is to prove a Hardy-type inequality
(\ref{dist}) for an arbitrary  $p>1$ for functions from $H^1,$ vanishing on small
alternating pieces of the boundary of the domain. It is assumed for
the simplicity  that $\Omega$ is a rectangle in $\mathbb{R}^2.$

Such a result is completely new in the theory of Hardy-type
inequalities and it gives us possibility to prove the boundedness of the sequence of functions considered in
microinhomogeneous domains. Indeed, in the classical statement the Hardy inequality was proved only for functions vanishing on the whole boundary of the domain. Our research shows that this requirement is too much restrictive. The Hardy inequality holds true even if the function vanishes on the set of small measure.

The paper \cite{Kormia} treats the case of weighted Hardy inequality for $p=2$ for functions vanishing on small pieces of the boundary.
A related result with weight functions equaled to 1 were  obtained earlier in \cite{ckp} and in
\cite{ckmp} for Friedrichs inequality. This inequality can be regarded as a
special case of weighted Hardy-type inequalities.

The current research is organized as follows: in Section \ref{s1} we give some
necessary definitions and formulate auxiliary lemmas. The main
results are presented and proved in Section \ref{s2} and Section
\ref{s3} is reserved for some concluding remarks.

 % Example of Definition
 
 \section{Preliminaries.}\label{s1}
Let $\Omega\subset\mathbb{R}^2$ be the rectangle $[0,a]\times[0,b]:$
$$\Omega=\{(x_1,x_2): 0\leq x_1\leq a, 0\leq x_2\leq b\}.$$
Assume that $\varepsilon>0$ is a small positive parameter,
$\varepsilon=\frac{b}{N},\ N\gg 1,\ N\in\mathbb{N},$ and
$\delta=const,$ $0<\delta<1.$ Moreover, let
$${{\Gamma}}:=\{(x_1,x_2)\in\partial\Omega:\  x_1=0\}.$$
We suppose that ${\Gamma}$ is represented in the form:
$${\Gamma}=\overline{{\Gamma}_{\varepsilon}}
\cup{\gamma}_{\varepsilon},\,\,\, {\Gamma}_{\varepsilon}=
\bigcup\limits_{{i}}({\Gamma}^{\varepsilon}_{i}),\,\,\,
{\gamma}_{\varepsilon}=\bigcup\limits_{{i}}({\gamma}^{\varepsilon}_{i}),
\,\,\, {\Gamma}^{\varepsilon}_{i}\cap{\gamma}^{\varepsilon}_{i}
=\varnothing,
$$
$$
|{\Gamma}^{\varepsilon}_{i}| =\varepsilon\delta,\quad
|\left({\Gamma}^{\varepsilon}_{i}
\cup{\gamma}^{\varepsilon}_{i}\right)|=\varepsilon,
$$
where  ${\Gamma}^{\varepsilon}_{i}$ and ${\gamma}^{\varepsilon}_{i}$
are alternating (see Figure \ref{f1}).

\begin{figure}[htb]
\begin{center}
\includegraphics[width=12cm]{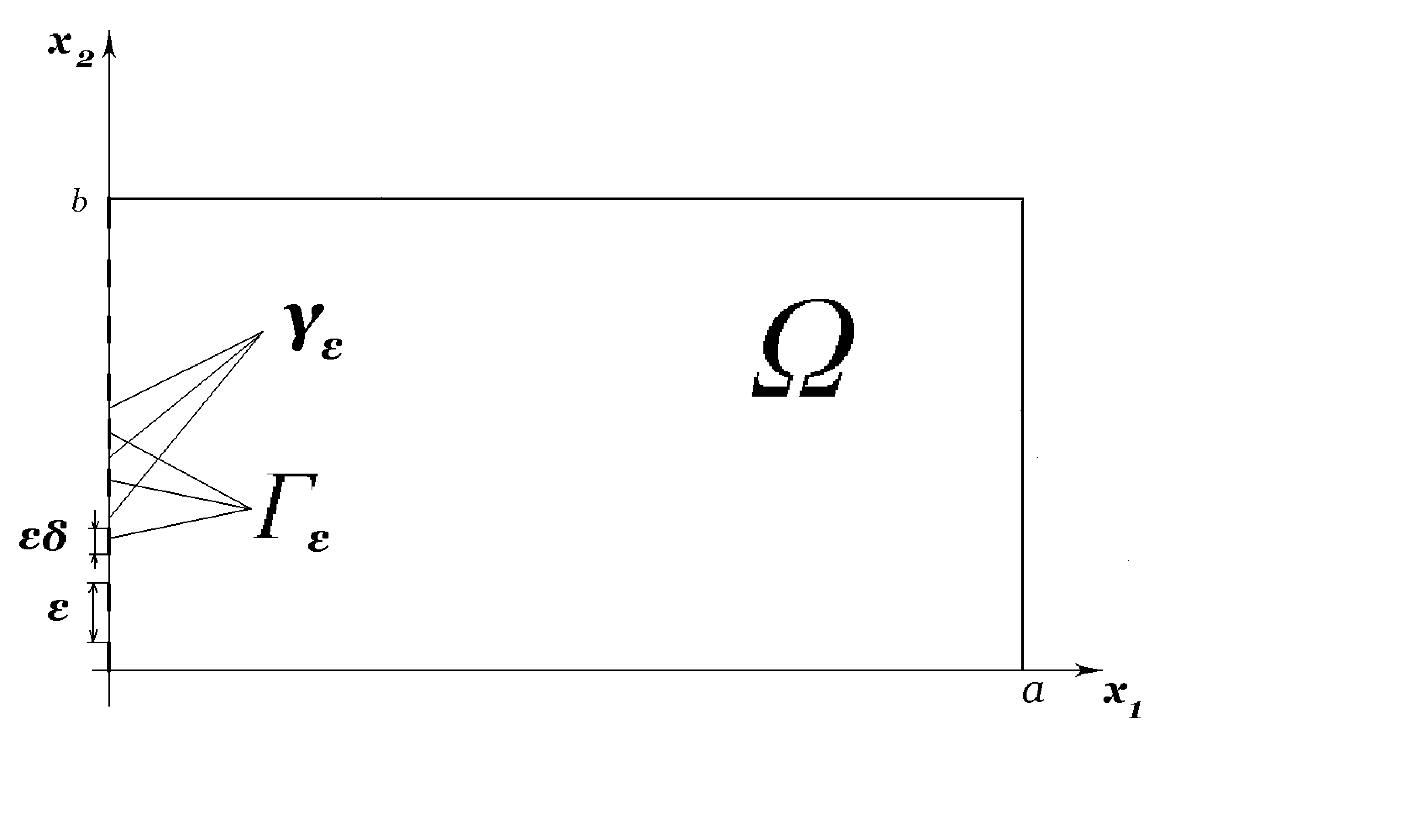}
\end{center}
\hspace{-2cm}
\caption{ The domain $\Omega$. }\label{f1}
\end{figure}

Denote by
$$\Pi_1^i:=\{(x_1,x_2)\in \Omega : 0\leq x_1\leq a,\, x_2\in{\Gamma}^{\varepsilon}_{i}\},$$
$$\Pi_2^i:=\{(x_1,x_2)\in \Omega : 0\leq x_1\leq a,\, x_2\in{\gamma}^{\varepsilon}_{i}\},$$
$$\Pi_1^i\cap\Pi_2^i=l_i; \Pi_1=\bigcup\limits_{i}\Pi_1^i,\ \Pi_2=\bigcup\limits_{i}\Pi_2^i.$$
%Fix a parameter $\theta>0.$ Define the set
%$\Omega^{\theta}:=\{x=(x_1, x_2)\in\Omega:\ x_1>\theta\}.$ The sets
%$\Pi_1^{i,\theta}, \Pi_2^{i,\theta}, \Pi_1^\theta \mbox{ and
%}\Pi_2^\theta$ are defined analogously.

Moreover, we use the
notation
$$B(x,r):=\{(y_1,y_2)\in \mathbb{R}^2: (y_1-x_1)^2+(y_2-x_2)^2\leq
r^2\},$$ and the average value of the function $u$ over $B(\cdot,
r)\in \mathbb{R}^2$ is defined as
$$u_B:=\frac{1}{\pi r^2}\int\limits_{B}u(x)\,dx.$$
Let $u$ be a locally integrable function on $\mathbb{R}^2.$ The
maximal functions $M(u)$ and $M_{R}(u)$ of $u$ are defined by
\begin{equation}\label{Mu}
M(u)(x):=\sup\limits_{r>0}u_B,\ M_R(u)(x):=\sup\limits_{0<r\leq R}u_B.
\end{equation}
Define the Sobolev space
$$W_{p}^{1}(\Omega, \Gamma_\varepsilon)=\{u_\varepsilon\in W_{p}^{1}(\Omega):\
u_\varepsilon|_{\Gamma_\varepsilon}=0 \}.$$ Analogously,
$$C^{\infty}(\Omega, \Gamma_\varepsilon)=\{u_\varepsilon\in C^{\infty}(\Omega):
u_\varepsilon=0 \mbox{ in a neighborhood of } {\Gamma_\varepsilon}=0
\}.$$ Let $x\in \Omega,\ \rho(x)={\rm dist}(x,\Gamma).$  We use the
following functions:
$$r_1(x)=\rho(x),\
r_2(x)={\rm
dist}(x,\,\Gamma_\varepsilon):=\inf\limits_{y\in\Gamma_\varepsilon}{\rm
dist}(x,y).
$$
According to the geometrical construction of the domain, (see Figure
\ref{hardy5.11}), $$r_1(x)<
r_2(x)<\rho(x)+(1-\delta)\frac{\varepsilon}{2}.$$
%\sqrt{\rho^2(x)+\frac{(\varepsilon(1-\delta))^2}{4}}

\begin{figure}[htb]
\begin{center}
\includegraphics[width=12cm]{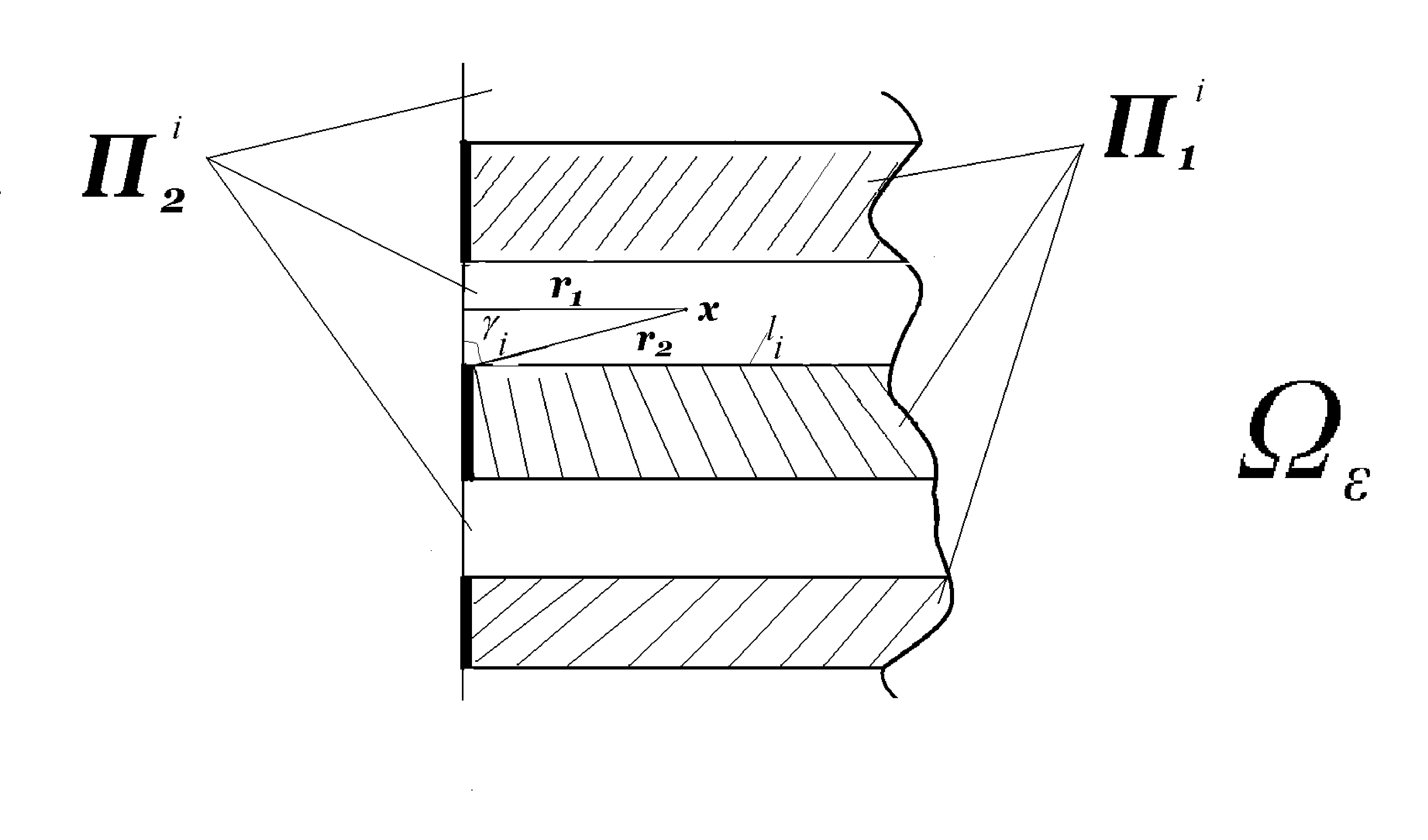}
\end{center}
\caption{ The domain $\Omega$. }\label{hardy5.11}
\end{figure}

We prove first an auxiliary Lemma on the validity of a Friedrichs inequality in the considered domain.
\begin{lemma}
Let $u_\varepsilon\in W_{p}^{1}(\Omega, \Gamma_\varepsilon).$ Then the
Friedrichs type inequality
\begin{equation}\label{Fr}
\int\limits_{\Pi_2}|u_\varepsilon|^p\,dx\leq
K(a,\varepsilon,\delta)\int\limits_{\Omega}|\nabla
u_\varepsilon|^p\,dx,
\end{equation}
holds with
$K(a,\varepsilon,\delta)=2^p(p+1)\left(a^{\frac{p}{q}+1}\frac{1-\delta}{\delta}+\varepsilon^{\frac{p}{q}+1}(1-\delta)^{\frac{p}{q}+1}\right).$
\end{lemma}
\begin{proof}
Fix the point $(x_1,x_2)\in\Pi_1^i.$ By using the Newton-Leibnitz
formula and H$\ddot{o}$lder inequality, we have
$$u_\varepsilon(x_1,x_2)=u_\varepsilon(x_1,x_2)-u_\varepsilon(0,x_2)=\int\limits_{0}^{x_1}
\frac{\partial u_\varepsilon}{\partial
x_1}\,dx_1\leq\int\limits_{0}^{a}\frac{\partial
u_\varepsilon}{\partial x_1}\,dx_1.$$ Hence,
$$|u_\varepsilon|^p(x_1,x_2)\leq\left(\int\limits_{0}^{a}\frac{\partial
u_\varepsilon}{\partial x_1}\,dx_1\right)^p\leq
\int\limits_{0}^{a}|\nabla u_\varepsilon|^p\,dx_1\left(\int\limits_{0}^{a}\,dx_1\right)^{\frac{p}{q}}=a^{\frac{p}{q}}\int\limits_{0}^{a}|\nabla u_\varepsilon|^p\,dx_1.$$ Then, by
integrating the last inequality with respect to $x_2$ and after that
with respect to $x_1$ over $\Pi_1^i,$ we obtain that
\begin{equation}\label{1i}
\int\limits_{\Pi_1^i}u_\varepsilon^p\,dx\leq
a^{\frac{p}{q}+1}\int\limits_{\Pi_1^i}|\nabla u_\varepsilon|^p\,dx.
\end{equation}
Now choose the point $(x_1,\widetilde{x}_2)\in\Pi_2^i$ such that
$\widetilde{x}_2=\frac{1-\delta}{\delta}x_2+\varepsilon\delta.$ This
means that $(x_1,\widetilde{x}_2)\in\Pi_2^i$ if and only if
$(x_1,{{x}_2})\in\Pi_1^i.$ Taking into account an evident inequality
$$(a+b)^p\leq 2^p(p+1)(a^p+b^p),\ \text{ for } p>2$$
and again using the Newton-Leibnitz
formula, we fined that
$$u_\varepsilon(x_1,\widetilde{x}_2)=u(x_1,x_2)+\int\limits_{x_2}^{\widetilde{x}_2}
\frac{\partial u_\varepsilon}{\partial x_2}\,d\widetilde{x}_2.$$
Consequently, $$
\begin{aligned}&|u_\varepsilon|^p(x_1,\widetilde{x}_2)\leq 2^p(p+1)|u_\varepsilon|^p(x_1,x_2)+2^p(p+1)\left(\int\limits_{x_2}^{\widetilde{x}_2}
\frac{\partial u_\varepsilon}{\partial x_2}\,d\widetilde{x}_2\right)^p\leq\\
&2^p(p+1)|u_\varepsilon|^p(x_1,x_2)+2^p(p+1)(\widetilde{x}_2-x_2)^{\frac{p}{q}}\int\limits_{x_2}^{\widetilde{x}_2}|\nabla
u_\varepsilon|^p\,d\widetilde{x}_2.
\end{aligned}$$ Integrating the last
inequality over $\Pi_2^i$ and substituting the integral on the
right-hand side by the greater integral, we get that

$$\int\limits_{\Pi_2^i}|u_\varepsilon|^p\,dx_1\,d\widetilde{x}_2\leq 2^p(p+1)\frac{1-\delta}{\delta}
\int\limits_{\Pi_1^i}|u_\varepsilon|^p\,dx_1\,dx_2+2^p(p+1)\varepsilon^{\frac{p}{q}+1}(1-\delta)^{\frac{p}{q}+1}\int\limits_{\Pi_2^i}|\nabla
u_\varepsilon|^p\,dx_1\,d\widetilde{x}_2.$$ Finally, by applying the
estimate (\ref{1i}) to the first integral on the right-hand side and
substituting both integrals by the greater integral, we obtain that
\begin{equation}\label{2i}
\int\limits_{\Pi_2^i}|u_\varepsilon|^p\,dx\leq
2^p(p+1)\left(a^{\frac{p}{q}+1}\frac{1-\delta}{\delta}+\varepsilon^{\frac{p}{q}+1}(1-\delta)^{\frac{p}{q}+1}\right)\int\limits_{\Pi_1^i\cup\Pi_2^i}|\nabla
u_\varepsilon|^p\,dx.
\end{equation}
By summarizing up the inequalities (\ref{1i}) and (\ref{2i}) with
respect to $i,$ we obtain the desired estimate:
\begin{equation*}
\begin{aligned}
\int\limits_{\Pi_2}|u_\varepsilon|^p\,dx=\int\limits_{\bigcup\limits_{i}\Pi_2^i}|u_\varepsilon|^p\,dx\leq
2^p(p+1)\left(a^{\frac{p}{q}+1}\frac{1-\delta}{\delta}+\varepsilon^{\frac{p}{q}+1}(1-\delta)^{\frac{p}{q}+1}\right)
\int\limits_{\bigcup\limits_{i}\left(\Pi_1^i\cup\Pi_2^i\right)}|\nabla
u_\varepsilon|^p\,dx=\\=2^p(p+1)\left(a^{\frac{p}{q}+1}\frac{1-\delta}{\delta}+\varepsilon^{\frac{p}{q}+1}(1-\delta)^{\frac{p}{q}+1}\right)
\int\limits_{\Omega}|\nabla u_\varepsilon|^p\,dx.
\end{aligned}
\end{equation*}
\end{proof}

We also need the following well-known Lemmas
\begin{lemma}\label{ziml}
Let $u\in W_1^{1}(B).$ Then
$$|u(x)-u_B|\leq 2\int\limits_{B}\frac{|\nabla u(y)|}{|x-y|}\,dy.$$
\end{lemma}
For the proof see in \cite[Lemma 7.16]{gilb}. The following important inequality was derived in \cite{Hajl}.
\begin{lemma}\label{le0}
Let $B=B(x,r)\subset\Omega,\ u\in C^{\infty}(\Omega),\ \Gamma\subset
\partial\Omega.$ Then
$$\inf\limits_{y\in \Gamma\cap B}\int\limits_B\frac{|\nabla
u(z)|}{|y-z|}\,dz\leq\int\limits_B\frac{|\nabla u(z)|}{|x-z|}\,dz.$$
\end{lemma}

\begin{lemma}\label{ziml2}
If $0<\alpha<2$ and $r>0,$ then there exists a constant $C>0,\,
C\leq\left(\frac{1}{2}\right)^{\alpha-2},$ such that for each
$x\in \mathbb{R}^2,$
$$\int\limits_{B(x,r)}\frac{|u(y)|}{|x-y|^{2-\alpha}}\,dy\leq C r^\alpha M(u)(x).$$
\end{lemma}
For the proof we refer to \cite[Lemma 2.8.3]{ziemer}. We will use
this result for the case $\alpha=1.$
%consequently, the constant
%$C=8\pi.$
Here, as usual, $M(u)$ stands for the Hardy-Littlewood maximal
operator defined in (\ref{Mu}). Moreover, the following important theorem (the
Hardy-Littlewood theorem on Maximal Operator) will be used:
\begin{theorem}\label{HL}
If $u\in L_p(\mathbb{R}^2),\ p>1,$ then there exists a constant
$\mathrm{C}>0 $
%=\frac{p}{p-1}2^p5^n$
such that
$$\|M(u)\|_p\leq C\|u\|_p.$$
\end{theorem}
For the proof see e.g. in \cite[Theorem 2.8.2]{ziemer}.

\section{The main results.}\label{s2}
Consider the function
\begin{equation}\label{max}
\rho_\varepsilon(x)=\begin{cases} \rho(x), &\text{if
$x\in\Pi_1$,}\\
\rho(x)+(1-\delta)\frac{\varepsilon}{2}, &\text{if $x\in\Pi_2.$}
\end{cases}
%\begin{aligned} \max\,(r_1(x),
%r_2(x))<\rho(x)+(1-\delta)\frac{\varepsilon}{2},\\
%M_{r_1}(f)<M_{r_2}(f)<M_{\rho(x)+(1-\delta)\frac{\varepsilon}{2}}(f).
%\end{aligned}
\end{equation}

Our first main result is the following pointwise inequality:
\begin{theorem}\label{le1}
Let $u_\varepsilon\in C^{\infty}(\Omega, \Gamma_\varepsilon).$  Then
there exists a constant $C,\ C\leq 4,$ such that the pointwise
inequality
\begin{equation}\label{pwH}
|u_\varepsilon(x)|\leq
C\rho_\varepsilon(x)M_{\rho_\varepsilon(x)}(|\nabla
u_\varepsilon|\chi_{B\left(\overline{x},
\rho_\varepsilon(x)\right)}(x))
\end{equation}
holds, where $\overline{x}\in\Gamma$ satisfies
$|x-\overline{x}|=\rho(x).$
\end{theorem}

\begin{proof}%[Proof of Theorem {\rm\ref{le1}}.]
Choose the point $x\in\Omega$ and denote by
$B:=B(\overline{x},\rho_\varepsilon(x)),$ where
$\rho_\varepsilon(x)$ is defined in (\ref{max}).

Then $B\cap\Gamma_\varepsilon\neq{\o}$ for each $x\in\Omega.$ Extend
the function $u_\varepsilon$ in
$\mathbb{R}^2\setminus\overline{\Omega}$ by reflecting it across the
boundary. By applying Lemma \ref{ziml} to the extended
$u_\varepsilon,$ we have for any $y\in B\cap\Gamma_\varepsilon:$
\begin{equation}\label{pr11}
\begin{aligned}
|u_\varepsilon(x)|=|u_\varepsilon(x)-u_\varepsilon(y)|\leq|u_\varepsilon(x)-{u_\varepsilon}_B|+|u_\varepsilon(y)-{u_\varepsilon}_B|\leq\\
\leq 2\left(\int\limits_B\frac{|\nabla
u_\varepsilon(z)|}{|x-z|}\,dz+\int\limits_B\frac{|\nabla
u_\varepsilon(z)|}{|y-z|}\,dz\right).
\end{aligned}
\end{equation}

Hence, by applying Lemma \ref{le0} to (\ref{pr11}), we obtain that
\begin{equation}\label{pr2}
%\begin{aligned}
|u_\varepsilon(x)|\leq 4\int\limits_B\frac{|\nabla
u_\varepsilon(z)|}{|x-z|}\,dz.
%\end{aligned}
\end{equation}
Finally, taking into account Lemma \ref{ziml2} and (\ref{max}), we
have that
\begin{equation}\label{pr1}
\begin{aligned}
|u_\varepsilon(x)|\leq
4c_2\,\rho_\varepsilon(x)\,M_{\rho_\varepsilon(x)}(|\nabla
u_\varepsilon|)(x)=\\=
\begin{cases}8\rho(x)\,M_{\rho(x)}(|\nabla
u_\varepsilon|), &\text{\!\!$x\in\Pi_1\cap
B\left(\overline{x},\rho(x)\right),$}\\
8\!\left(\rho(x)\!+\!\frac{(1-\delta)}{2}\varepsilon\right)
\!M_{\rho(x)+\frac{(1-\delta)}{2}\varepsilon}(|\nabla
u_\varepsilon|), &\text{\!\!$x\in\Pi_2\cap
B\!\left(\overline{x},\rho(x)\!+\!\frac{(1-\delta)}{2}\varepsilon\right).$}
\end{cases}
\end{aligned}
\end{equation}
The proof is complete.
\end{proof}

The next theorem generalizes the classical
Hardy-type inequalities to a much more wide class of
functions.
\begin{theorem}\label{main}
Let $\rho_\varepsilon(x)$ be the function defined in
{\rm(\ref{max})} and  $p>1,$ $0\leq\alpha<\alpha_0.$
 Then the estimate
\begin{equation}\label{Hardy}
\int\limits_{\Omega}\rho_\varepsilon^{-p+\alpha}(x)|u_\varepsilon|^p(x)\,dx\leq
C\int\limits_{\Omega}\rho_\varepsilon^{\alpha}(x)|\nabla
u_\varepsilon(x)|^p\,dx
\end{equation}
holds for each fixed $\varepsilon$ for all functions
$u_\varepsilon\in W^{1}_{p}(\Omega, \Gamma_\varepsilon),$ where the
constant $C$ does not depend on $u_\varepsilon$ and on
$\varepsilon.$ %($C_1\leq 1600$).
\end{theorem}

\begin{proof}%[Proof of Theorem {\rm\ref{main}}.]
Fix $u_\varepsilon\in C^{\infty}(\Omega, \Gamma_\varepsilon)$
extended in $\mathbb{R}^n\setminus{\overline{\Omega}}.$ According to
(\ref{pwH}) the inequality
$$\frac{|u_\varepsilon(x)|}{\rho_\varepsilon(x)}\leq 8 M_{\rho_\varepsilon(x)}\left(|\nabla
u_\varepsilon|\chi_{B(\overline{x},
\rho_\varepsilon(x))}\right)(x)$$ holds for all $x\in \Omega.$ Then
we have that
$$\int\limits_{\Omega}\frac{|u_\varepsilon(x)|^p}{\rho^p_\varepsilon(x)}\,dx\leq 8^p\int\limits_{\Omega}M^p_{\rho_\varepsilon(x)}
\left(|\nabla u_\varepsilon|\chi_{B\left(\overline{x},
\rho_\varepsilon(x)\right)}(x)\right)\,dx.$$ The statement in
Theorem \ref{le1} implies that
$$\int\limits_{\Omega}M^p_{\rho_\varepsilon(x)}\left(|\nabla u_\varepsilon|\chi_{B\left(\overline{x},
\rho_\varepsilon(x)\right)}\right)(x)\,dx\leq
C_1\int\limits_{\Omega}|\nabla u_\varepsilon(x)|^p\,dx.$$ Thus, it
yields that
$$\int\limits_{\Omega}\left(\frac{|u_\varepsilon(x)|}{\rho_\varepsilon(x)}\right)^p\,dx
\leq C\int\limits_{\Omega}|\nabla u_\varepsilon(x)|^p\,dx,$$ where
$C=8^pC_1.$

Hence, the inequality (\ref{Hardy}) holds with $\alpha=0.$ The next
step is to prove (\ref{Hardy}) for $\alpha>0.$ Choose $\sigma>0$ and
put $v_\varepsilon=|u_\varepsilon|\rho_\varepsilon^{\sigma}.$ Taking into account an evident inequalities $(a+b)^2\leq 2(a^2+b^2),$
$(a+b)^p\leq 2^p(p+1)(a^p+b^p),\ \text{ for } p>2$ it is
not difficult to derive that
\begin{equation}\label{put}
\begin{aligned}
|\nabla v_{\varepsilon}|^p=(|\nabla v_{\varepsilon}|^{2})^\frac{p}{2}\leq\left(2\left(\frac{\partial
u_\varepsilon}{\partial x_1}\right)^2\rho_\varepsilon^{2\sigma}+2\sigma^2\rho_\varepsilon^{2\sigma-2}\left(\frac{\partial
\rho_{\varepsilon}}{\partial x_1}|u_\varepsilon|\right)^2+\rho_\varepsilon^{2\sigma}\left(\frac{\partial
u_\varepsilon}{\partial x_2}\right)^2\right)^{\frac{p}{2}}\\
\leq2^{p}\left(\frac{p}{2}+1\right)\left(|\nabla u_\varepsilon|^p\rho_\varepsilon
^{\sigma p}+\sigma^p\rho_\varepsilon^{(\sigma-1)p}|u_\varepsilon|^p\right).
\end{aligned}
\end{equation}
By now applying (\ref{Hardy}) with $\alpha=0$ to $v_{\varepsilon},$ we obtain that
\begin{equation*}
\begin{aligned}
\int\limits_{\Omega}\rho_\varepsilon^{-p+\sigma p}|u_\varepsilon|^p\,dx
\leq C_2\left(\int\limits_{\Omega}|\nabla
u_\varepsilon|^p\rho_\varepsilon^{\sigma p}\,dx+\sigma^p\int\limits_{\Omega}\rho_\varepsilon^{p(\sigma-1)}|u_\varepsilon|^p\,dx
\right).
\end{aligned}
\end{equation*}
Here $C_2=C 2^p\left(\frac{p}{2}+1\right).$
 If $1-C_2\sigma^p>0,$ then we have
$$\int\limits_{\Omega}\rho_\varepsilon^{-p+p\sigma}|u_\varepsilon|^p\,dx\leq
\frac{C_2}{1-C_2\sigma^p}\int\limits_{\Omega}\rho_\varepsilon^{p\sigma}|\nabla
u_\varepsilon|^p\,dx.$$
 Substituting $\sigma$ by $\frac{\alpha}{p},$ we prove the inequality (\ref{Hardy}) for $\alpha>0.$  Finally, by approximating the functions from $W^{1}_{p}(\Omega,
\Gamma_\varepsilon)$ by smooth functions belonging to
$C^{\infty}(\Omega, \Gamma_\varepsilon),$ we can complete the proof.
\end{proof}
Our final main result is presented in the next theorem.
\begin{theorem}\label{main2}
Let $\rho(x)=dist(x,\Gamma)$ and $p>1,$ $0\leq\alpha<\alpha_0.$ Then there exists a constant $C>0$ such that
\begin{equation}\label{Hardyro}
\int\limits_{\Omega}\rho^{-p+\alpha}(x)|u_\varepsilon|^p(x)\,dx\leq
C\int\limits_{\Omega}\rho^{\alpha}(x)|\nabla
u_\varepsilon(x)|^p\,dx
\end{equation}
holds for all functions $u_\varepsilon\in
W^1_p(\Omega, \Gamma_\varepsilon).$
%where the constant $$C(a,
%\varepsilon, \delta,
%\theta)=4+2\frac{1-\delta}{\theta^{2}}\left(\frac{a^2}{\delta}+\varepsilon^2(1-\delta)\right).$$
\end{theorem}

\begin{proof}
Due to the geometrical construction of the domain, $$\rho(x)\leq\rho_{\varepsilon}(x)<\rho(x)+(1-\delta)\frac{\varepsilon}{2}.$$
Having in mind the estimate (\ref{Hardy}), and an evident inequality for nonnagative $A~\geq~0, B~\geq~0,$ $(A+B)^{\alpha}\leq 2^{\alpha}(A\cdot B)^{\frac{\alpha}{2}}$ one concludes the following:

\begin{equation}\label{Hardyromain}
\int\limits_{\Omega}\frac{|u_\varepsilon|^p(x)}{(\rho(x)+(1-\delta)\frac{\varepsilon}{2})^{p-\alpha}}\,dx\leq
C_1\int\limits_{\Omega}\rho^{\alpha}(x)|\nabla
u_\varepsilon(x)|^2\,dx
\end{equation}
 where $$C_1=2^{\frac{\alpha}{2}}(1-\delta)^{\frac{\alpha}{2}}{\varepsilon^{\frac{\alpha}{2}}}C.$$
Let $\gamma_i$ be the angle between the segment $\Gamma_i$ and the line $XA_i$ where $XA_i=\text{dist}(x, \Gamma^\varepsilon_i).$
Then $$\rho(x)=\begin{cases}
\text{dist} {(x,\Gamma^\varepsilon)}, \text{ if } x=(x_1,x_{2})\in l_i;\\
(1-\delta)\frac{\varepsilon}{2}tg\gamma_i<\infty, \text{ else }.
\end{cases}$$
Therefore, the estimate (\ref{Hardyromain}) can be rewritten as
\begin{equation}\label{Hardyromain2}
\int\limits_{\Omega}\frac{|u_\varepsilon|^p(x)}{\rho(x)^{p-\alpha}}\,dx\leq
C_2\int\limits_{\Omega}\rho^{\alpha}(x)|\nabla
u_\varepsilon(x)|^2\,dx,
\end{equation}
where $$C_2=C_1\left(1+(1-\delta)\frac{\varepsilon}{2}ctg\gamma_i\right)^{p-\alpha}.$$

\end{proof}

\section{Concluding remarks.}\label{s3}
\begin{remark}\label{rem1}{\rm
The condition $\delta=1$ in the definition (\ref{max}) corresponds
to the case $\Gamma_\varepsilon~=~{\Gamma}.$ In this case the
Hardy-type inequality {\rm(\ref{Hardy})} becomes the inequality
{\rm(\ref{dist})} and the constant in {\rm(\ref{Hardyro})} equals to
the constant in the classical Hardy inequality when the function is
vanishing on the whole boundary.}
\end{remark}

\begin{remark}\label{rem2}{\rm
If $\alpha=0,$ then {\rm(\ref{Hardy})} takes the  form
\begin{equation}\label{Hardy1}
\int\limits_{\Omega}\left(\frac{|u_\varepsilon(x)|}{\rho_\varepsilon(x)}\right)^p\,dx\leq
C\int\limits_{\Omega}|\nabla u_\varepsilon(x)|^p\,dx.
\end{equation}
This inequality is an important tool for proving the embedding theorems.}
%while {\rm(\ref{Hardyro})} becomes
%$$\int\limits_{\Omega^\theta}\left(\frac{u_\varepsilon(x)}{\rho(x)}\right)^2\,dx\leq
%C(a, \varepsilon, \delta, \theta)\int\limits_{\Omega}|\nabla
%u_\varepsilon(x)|^2\,dx.$$
\end{remark}

\begin{remark}{\rm
The proved result is valid for the case when $\Omega$ is an arbitrarily Lipschitz domain and the function vanishes on the small pieces of its boundary.}
\end{remark}
%We conjecture that these Hardy-type inequalities holds also when
%$p=2$ is replaced by any $p>1$ but then another type of proof must
%be found.
%
%\begin{remark}\label{rem3}
%In this paper we have succeeded to prove a weighted Hardy-type
%inequality in a fixed domain for functions vanishing on a part of
%the boundary. We can see several open questions equipped with this
%result. For instance, one interesting problem is to try to find a
%weighted Hardy-type inequality for perforated domains in the case
%when the size of perforation depends on the small parameter.
%\end{remark}

 % Acknowledgments
\section*{\large Acknowledgments} The author thanks Professors Gregory Chechkin and Lars-Erik Persson for introducing me in the area of functional inequalities.

\def\bibname{\vspace*{-30mm}{

\vskip 1 cm \footnotesize
\begin{flushleft}
Yulia Koroleva \\ % name of the authors from the same institutio
 Moscow Tikhonov Institute \\
 of Electronics and Mathematics,\\ % name of the department, where author works
HSE National Research University, \\ % name of the university, where author works
Moscow Tallinskaya str. 34, Russia\\ % buseness address % buseness address 
E-mails: yo.koroleva@hse.ru
\end{flushleft}

%\vskip0.5cm
%\begin{flushright}
%Received: ??.??.20??
%\end{flushright}

\end{document}